\documentclass[12pt]{article}

\usepackage[latin1]{inputenc}
\usepackage{amsmath,amssymb}
\usepackage{latexsym}
\usepackage{amsthm}
\usepackage{color,graphicx}

\newtheorem{theorem}{Theorem}[section]
\newtheorem{corollary}[theorem]{Corollary}
\newtheorem{lemma}[theorem]{Lemma}
\newtheorem{proposition}[theorem]{Proposition}
\newtheorem{definition}[theorem]{Definition}

\newtheorem{remark}[theorem]{Remark}
\newtheorem{question}[theorem]{Question}

\numberwithin{equation}{section}

\def\square{{\vcenter{\vbox{\hrule height.3pt
        \hbox{\vrule width.3pt height5pt \kern5pt
           \vrule width.3pt}
        \hrule height.3pt}}}}

\def\sS {{\cal S}}

\def\wt{\widetilde}

\def\ol{\overline}

\def\bee{\begin{equation}}
\def\bet{\begin{theorem}}
\def\bep{\begin{proposition}}
\def\bef{\begin{proof}}
\def\bel{\begin{lemma}}
\def\bec{\begin{corollary}}
\def\bed{\begin{definition}}
\def\ber{\begin{remark}}
\def\eee{\end{equation}}
\def\eet{\end{theorem}}
\def\eep{\end{proposition}}
\def\eef{\end{proof}}
\def\eel{\end{lemma}}
\def\eec{\end{corollary}}
\def\eed{\end{definition}}
\def\eer{\end{remark}}

\def\R{{\mathbb R}}

 \def\qq {\qquad}

\def\wt{\widetilde}
\def\ol{\overline}

\def\square{{\vcenter{\vbox{\hrule height.3pt
        \hbox{\vrule width.3pt height5pt \kern5pt
           \vrule width.3pt}
        \hrule height.3pt}}}}

\def\tlint{{- \kern-0.85em \int \kern-0.2em}}  
\def\dlint{{- \kern-1.05em \int \kern-0.4em}}  

\def\sS {{\cal S}}

\begin{document}

\title{Two results from Mandelbaum's paper: ``The dynamic complementarity problem''} 
\author{Richard F. Bass}

\date{\today}

\maketitle

\begin{abstract}
A draft of a paper by Mandelbaum, ``The dynamic complementarity problem,'' was 
circulated in 1987, but has never been published. We give an exposition of two 
important results from that paper which are not readily accessible in the literature. 

The first is an example of a Skorokhod problem in two dimensions in the quadrant for which there is not uniqueness. The second is a proof of uniqueness for the Skorokhod problem in two dimensions in the quadrant in a critical case.
%
%
\end{abstract}


\section{Introduction}\label{sect-intro}

In 1987  A.\ Mandelbaum circulated a draft of a paper titled
``The dynamic complementarity problem.'' Although this paper has been cited many
times in the ensuing decades, it has never been published. Of particular interest to many is his example of a 
Skorokhod problem for which a solution exists but 
for which uniqueness does not hold. 

Although this example is of great importance, to the best of our knowledge the only publicly available exposition is in the Ph.D.\ dissertation of Whitley \cite{Whit}, which is a bit difficult to find (and may require paying 
to get a copy). See also Stewart \cite{St}, which  uses similar methods to handle a related 
problem. We thought that it would be worthwhile to give an exposition of Mandelbaum's counterexample that is freely available on the Internet.

Less well known is another result in that paper. Consider the matrix
$R$ that appears in the Skorokhod problem (details in a moment).
Uniqueness has been proved  for the Skorokhod problem for a certain class of matrices 
$R$ by Harrison and Reiman \cite{HR} and Williams \cite{W95} 
pointed out that the proof works for a much  larger class of matrices; see below.
In his paper Mandelbaum proves uniqueness for a certain critical case.
See Remark \ref{uniq-summary} for a summary of the known results in two dimensions
and where things stand in higher dimensions.

We provide proofs for the counterexample and for the critical case.
We emphasize that these notes are expository and all the ideas are due to 
Mandelbaum. We thank him for providing us with a copy of his draft paper. We also would like to thank  K.\ Burdzy\ and R.\ Williams for many helpful conversations on the
subject of the Skorokhod problem.

Let us turn to describing the Skorokhod problem and the corresponding
Skorokhod equation. Except for Remark \ref{r-uniq}, for the remainder of these notes we consider the two-dimensional case only.

For a vector $b=(b_1,b_2)\in \R^2$, we say $b\ge 0$ if $b_1\ge 0$ and $ b_2\ge 0$.
Let $D=\{b\in \R^2: b\ge 0\}$. 

\begin{definition}\label{intro-SP} {A driving function $f$  is a continuous function
from $[0,\infty)$ to $\R^2$ with $f(0)\ge 0$. The Skorokhod problem is to find\\
(1) $g$ a continuous function from $[0,\infty)\to D$;\\
(2) $m$ a continuous  function on $[0,\infty)$ with $m(0)=0$ and
each $m_j(t)$ is non-decreasing, $j=1,2$;\\
\phantom{xxx} such that \\
(3) $g(t)=f(t)+Rm(t)$ for all $t\ge 0$;\\
\phantom{xxx}  and\\
(4) $m_j$ increases only when $g_j=0$, $j=1,2.$
}
\end{definition}

The equation (3) is known as the Skorokhod equation. 
It arises as a way to represent
reflecting Brownian motion when $f$ is a Brownian path.

Note that (4) is equivalent to
\begin{equation}\label{intro-dcp}
\int_0^\infty g_j(t)\, dm_j(t)=0, \qq j=1,2.
\end{equation}
Finding $g$ and $m$ satisfying (1)--(3) and \eqref{intro-dcp}
is one type of dynamic complementarity problem.

Mandelbaum is concerned with the following.

\begin{question}\label{Q1}{
For which matrices $R$ does there exist a unique solution to the 
Skorokhod problem for \textbf{every} driving
function $f(t)$ with $f(0)\ge 0$?}
\end{question}

One can also ask:
\begin{question}\label{Q2}{
For which matrices $R$ does there 
fail to be a unique solution to the Skorokhod problem
 when the driving function is a typical Brownian
path?}
\end{question}

We will not address Question \ref{Q2} in these notes except in Remark \ref{uniq-BM}.

Existence of a solution to the Skorokhod problem is thoroughly understood.
It is known (see, e.g., \cite{W95} or \cite{BK}) that there will exist at least one
solution for every driving function if and only if $R$ is a completely-$\sS$ matrix. In two dimensions this  means that the diagonal entries of $R$ are positive and that there exists
$x\ge 0$ such that $Rx > 0$. 

Suppose $R_{11}, R_{22}>0$. If we let $\wt m_i(t)=R_{ii}m_i(t)$ and $\wt R_{ij}=R_{ij}/R_{jj}$, $i,j=1,2$,
then the
equation $g=f+Rm$ can be rewritten as $g=f+\wt R\wt m$.
Therefore there is no loss of generality in assuming that the diagonal elements
of $R$ are equal to 1.

Suppose  from now on that
\begin{align}
R&=\begin{pmatrix}
        1 & a_1  \\
        a_2& 1 \\
        \end{pmatrix}.\label{intro-form}
\end{align}
It is easy to see that 
$R$ will be completely $\sS$ when 
 (1) at least one of $a_1,a_2$ is positive
 or when (2) $a_1a_2<1$. If neither (1) nor (2) hold, then $R$ will not be
completely-$\sS$.

As for uniqueness of the Skorokhod problem, 
let $Q=I-R$ and let $|Q|$ be the matrix with each coordinate of $Q$ replaced by its absolute value. Thus
$|Q|=\begin{pmatrix}
        0 & |a_1|  \\ 
        |a_2|& 0 \\
        \end{pmatrix}.$
A calculation shows that the spectral radius of $|Q|$ is $\pm \sqrt{|a_1a_2|}$.
When this spectral radius is strictly less than 1,
there is uniqueness for the Skorokhod problem for every driving function by
\cite{HR} as improved by \cite{W95}.

In Section \ref{sect-cter} we take $a_2=1$, $a_1=-2$, and show that there
exist two distinct solutions to the Skorokhod problem. The proof is Mandelbaum's, although we are able to streamline it a bit since we are in a special case of
his more general results.

In Section \ref{sect-uniq} we consider the critical case where $|a_1a_2|=1$,
$a_2>0$, and 
$a_1<0$. Mandelbaum's proof of uniqueness
is a quite brief sketch, so we flesh out the proof with additional details.

\section{A counterexample}\label{sect-cter}

We present Mandelbaum's example  of a deterministic version of the Skorokhod
problem in two dimensions where uniqueness does not hold. 
We mention that Bernard and el Kharroubi \cite{BK} gave an example in three dimensions where the driving function is linear.

Set
\begin{align}\notag
R&=\begin{pmatrix} 
        1 & -2  \\
        1& 1 \\
        \end{pmatrix}.
\end{align}

\begin{theorem}\label{cter-T1}
There exists a driving function $f(t)$ for which there exist two distinct solutions $g(t)$
to the Skorokhod equation
$$g(t)=f(t)+Rm(t).$$
\end{theorem}

\begin{proof}
We define an auxiliary function $u=(u_1,u_2)$ mapping $[0,1]$ to $\R^2$ as follows. Let $t_n=2^{-n}$, $n\ge 0$.
Set $u(0)=0$. For $k$ a non-negative integer, set
\begin{align*}
u(t_{4k})=(-2^{-2k}, 2^{-2k}), \qquad & u(t_{4k+1})=(-2^{-2k},-2^{-2k-1}),\\
u(t_{4k+2})=(2^{-2k-1}, -2^{-2k-1}), \qquad &u(t_{4k+3})=(2^{-2k-1},
2^{-2k-2}).
\end{align*}
For $t$ between $t_{n+1}$ and $t_{n}$ we linearly interpolate between 
$u(t_{n+1})$ and $u(t_{n})$.
Thus $u$ is a continuous piecewise linear function 
of bounded variation that starts at 0 and spirals out away
from the origin. Note that each $t_n$ is either on the line $u_1+u_2=0$ or
the line $u_1-2u_2=0$.
See Figure 1, which is a drawing of the $(u_1,u_2)$ plane with the graph of
$\{u(t): t_7\le t\le t_0\}$ shown in green.

\begin{figure}[ht]
\centering
\caption{Graph of $u(t)$}
\includegraphics[width=0.9\linewidth]{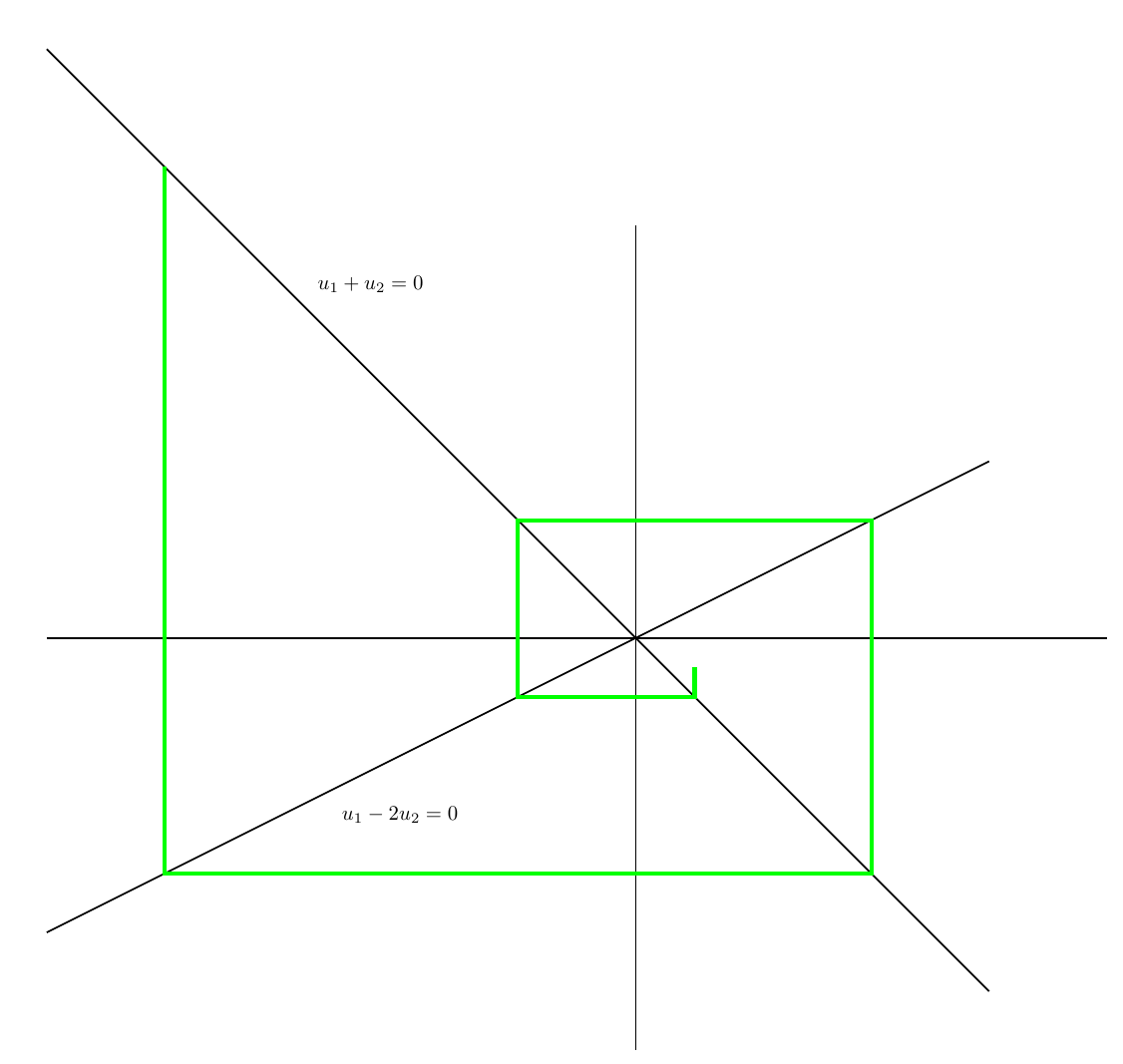}
\label{cter-fig1}
\end{figure}

For $j=1,2$, each $u_j$ is of bounded variation, and we write 
$$u_j(t)=m_j(t)-\ol m_j(t),$$
where $m_j, \ol m_j$ both are non-decreasing and start at 0. 
Define
\begin{equation}
f(t)=-(Rm(t)\land R \ol m(t)),\label{cter-f-def}
\end{equation}
where this equation means that $f_j(t)=-((Rm(t))_j\land (R\ol m(t))_j)$ for
$j=1,2$. 
Set 
\begin{align}
g(t)&=(Rm(t)-R\ol m(t))^+,\label{cter-g-def}\\
\ol g(t)&=(Rm(t)-R\ol m(t))^-,\notag
\end{align}
where again these equations are interpreted component-wise.

We claim $(g,f,m)$ and $(\ol g,f, \ol m)$ are two distinct solutions to the
Skorokhod equation.
Clearly $g, \ol g\ge 0$, $g\ne \ol g$, $m,\ol m$ both start at 0 and
are non-decreasing, and all the functions are continuous.
That $$g(t)=f(t)+Rm(t), \qq \ol g(t)=f(t)+R\ol m(t)$$
follow from the identities
$a-(a\land b)=(a-b)^+$ and $ b-(a\land b)=(a-b)^-.$

It remains to show that $m_j$ only increases when $g_j=0$, $j=1,2$, and 
the same with $\ol m_j, \ol g_j$. We do the case where $m_2$ increases, the other
three cases being exactly similar. $m_2$ increases only when $t$ is in
an interval $[t_{4k+1}, t_{4k}]$ for some non-negative integer $k$. In that 
interval $u$ moves vertically  from
the line $(Ru)_1=u_1-2u_2=0$ to the line $(Ru)_2=u_1+u_2=0$. 
Since  $Rm(t)-R\ol m(t)=Ru(t)$, for such $t$
$$(Rm(t))_2-(R\ol m(t))_2=(Ru(t))_2\le 0.$$
Therefore for such $t$, using \eqref{cter-g-def} we have $g_2(t)=0$ as required.
\end{proof}

\begin{remark}\label{cter-general}{\rm The same argument works if in the matrix $R$ we replace $-2$ by any real that is strictly less than $-1$.}
\end{remark}

\begin{remark}\label{cter-BV}{\rm In the above proof, for each interval
$[t_{n+1},t_n]$ only one of $u_1,u_2$ changes, and for whichever $u_j$ that
changes, either $u_j$ increases or decreases over the entire time interval.
It is thus straightforward to see
 how the decomposition of $u$ into the difference of non-decreasing
functions $m$ and $\ol m$ occurs.}
\end{remark}

\section{Uniqueness - the critical case}\label{sect-uniq}

We examine the critical case when the
spectral radius is exactly 1. 

\begin{lemma}\label{uniq-L1} Suppose $C>0$. There is a unique solution for every continuous
driving function for the deterministic Skorokhod problem with matrix $R=\begin{pmatrix}
        1 & a_1  \\
        a_2& 1 \\
        \end{pmatrix}$
if and only if there is a unique solution for every continuous driving function
for the Skorokhod problem with matrix $S=\begin{pmatrix}
        1 & Ca_1  \\
        a_2/C& 1 \\
        \end{pmatrix}$.
\end{lemma}

\begin{proof}
 If we write out $g=f+ Rm$ in coordinates and multiply the second equation by
$1/C$, we get
\begin{align*}
g_1&=f_1+ m_1+ a_{1}m_2,\\
\frac1{C}g_2&=\frac1{C}f_2+\frac1{C} a_{2}m_1+\frac1{C} m_2.
\end{align*}
Let $\wt m_1=m_1$, $\wt g_1=g_1$, $\wt f_1=f_1$, and
$$\wt m_2=\frac1{C}m_2, \quad \wt g_2=\frac1{C}g_2, \quad \wt f_2=\frac1{C}f_2.$$
We then have
$$\wt g=\wt f+ S \wt m.$$

It follows that there will be two distinct solutions to the Skorokhod problem for a driving
function $f$ with respect to the matrix $R$ if and only if there are  two distinct solutions
to the Skorokod problem with driving function $\wt f$ with respect to the
matrix $S$.
\end{proof}

\begin{theorem}\label{uniq-T1}
If $|a_1a_2|=1$, $a_2>0$, and $a_1<0$, there is a unique solution for every
driving function $f$ to the Skorokhod problem when $R$ is of the form \eqref{intro-form}.
\end{theorem}

\begin{proof} 
Letting  $C=1/a_2$ and applying Lemma \ref{uniq-L1} we see that it
suffices to look at $R$ given by
\begin{align} 
R&=\begin{pmatrix}
        1 & -1  \\
        1& 1 \\
        \end{pmatrix}. \label{uniq-form}
\end{align}
 
Suppose there are two solutions $(g,m)$ and $(\ol g, \ol m)$.
Let $u=m-\ol m$. Then
$$g-\ol g=Ru.$$

Using \eqref{intro-dcp}, for $j=1,2$
\begin{align*}
(Ru)_j\, du_j&=(g_j-\ol g_j)\, d(m_j-\ol m_j)\\
&=g_j\, dm_j-g_j\, d\ol m_j-\ol g_j\, dm_j+\ol g_j\, d\ol m_j\\
&=-g_j\, d\ol m_j-\ol g_j\, dm_j\\
&\le 0,
\end{align*}
where we write $d\mu\le 0$ for a signed measure $\mu$  if $\mu(B)\le 0$ for all 
Borel sets $B$, or equivalently, if $\mu$ has no positive part.

With our choice of $R$, this equation is the same as
\begin{align}
(u_1+u_2)\, du_2&\le 0 \label{E2}\\
(u_1-u_2)\, du_1&\le 0.\notag
\end{align}

Divide the $(u_1,u_2)$ plane into 4 pieces by using the lines $u_2=u_1$ and
$u_2=-u_1$. We label clockwise the pieces $N, E, S, W$ (for north, east, south, and west). To assign the boundaries, we define
\begin{align*}
N&=\{(u_1,u_2): u_2>0, -u_2< u_1\le u_2\},\\
E&=\{(u_1,u_2): u_1>0, -u_1< u_2\le u_1\},\\
S&=\{(u_1,u_2): u_2<0, u_2< u_1\le -u_2\},\\
W&=\{(u_1,u_2): u_1<0, u_1<u_2\le -u_1\}.
\end{align*}
Note that we include one side of the boundary of $N$ in $N$ but not 
the other.
This is true for each of the four sectors.

Let $v=\max(|u_1|,|u_2|)$. For $(u_1,u_2)\in N$ we see that $v=u_2$ and
$u_1+u_2>0$. Using the first line of \eqref{E2} this shows $dv=du_2\le 0$.

For $(u_1,u_2)\in W$
we have $-(u_1-u_2)=-u_1+u_2>0$, and using the second line of \eqref{E2}
we conclude $du_1\ge 0$, so $dv=-du_1\le 0$. Note that the ray $u_2=-u_1, u_2>0$ is included
in $W$ but not $N$.

We argue similarly for $E$ and $S$. Hence $v(0)=\max(|u_1(0)|,|u_2(0)|)=0$, $v(t)\ge 0$ for all $t$, 
and $dv\le 0$ on the set $\{t: v(t)\ne 0\}$. 
This means that $v$ is non-increasing on $\{t: v(t)\ne 0\}$, so $v$
must be identically 0.
Therefore  $u$ is identically 0,
so $m=\ol m$, and then $g=\ol g$.
\end{proof}

\begin{remark}\label{uniq-summary}{\rm The results concerning
Question \ref{Q1} in two dimensions are the following.
Suppose $R$ is given by \eqref{intro-form}.
There are five cases to consider:\\
(1) $|a_1a_2|<1$;\\
(2) $|a_1a_2|=1$, $a_1,a_2$ are of opposite signs;\\
(3) $|a_1a_2|=1$, $a_1,a_2$ are both positive;\\
(4) $|a_1a_2|>1$, $a_1,a_2$ are of opposite signs;\\
(5) $|a_1a_2|>1$, $a_1,a_2$ are both positive.\\
\noindent ($R$ will not be
completely-$\sS$ if $|a_1a_2|\ge 1$ and $ a_1,a_2$ are both negative.)

Uniqueness holds in Case (1) by \cite{HR,W95}. The results of Mandelbaum 
discussed above together with Lemma \ref{uniq-L1} take care of Cases (2) and (4).
K.~Burdzy and I \cite{BaBu2} have recently resolved Cases (3) and (5) (uniqueness for $g$ but not
$m$ in Case (3); non-uniqueness for Case (5)).
}
\end{remark}

\begin{remark}\label{r-uniq}{\rm For dimensions larger than 2, the result of
\cite{HR} and \cite{W95} still holds: if the spectral radius of $|Q|$ is 
strictly less than 1, uniqueness holds. 
As far as we know, the cases where the spectral radius is greater than or 
equal to  one are largely open.

It is not known if the  proof given in this section can be extended to higher dimensions for the case where the spectral dimension is exactly one.}
\end{remark}

\begin{remark}\label{uniq-BM}{\rm When there is uniqueness for every driving function, there will be uniqueness when $f$ is replaced by the path of 
a Brownian motion.
It is conceivable, however, that a matrix $R$ could be such that there is not uniqueness for every driving function, but that there is uniqueness almost surely
when $f$ is
a Brownian path. There is not much known here. See \cite{BaBu},
where it is shown that for a large class of matrices $R$ pathwise uniqueness
does not hold for almost every Brownian path.

There is a notion of weak uniqueness in probability theory which, not surprisingly, is weaker than the notion of pathwise uniqueness. Weak uniqueness holds for
every matrix $R$ which is completely--$\sS$; see Taylor and Williams \cite{TW} for 
definitions and proofs.}
\end{remark}

\bibliographystyle{alpha}

\end{document}